\documentclass[11pt,reqno]{amsart}
\usepackage[breaklinks]{hyperref}
\usepackage{url}
\usepackage{amssymb}

\newcommand{\tf}{\tfrac}

  \renewcommand{\a}{\alpha}
\renewcommand{\b}{\beta}

\newcommand{\G}{\Gamma}
\renewcommand{\l}{\lambda}

\renewcommand{\(}{\left\(}
\renewcommand{\)}{\right\)}
\renewcommand{\[}{\left\[}
\renewcommand{\]}{\right\]}

\numberwithin{equation}{section}
 \theoremstyle{plain}
\newtheorem{theorem}{Theorem}[section]
\newtheorem{lemma}[theorem]{Lemma}
\newtheorem{corollary}[theorem]{Corollary}

\newtheorem{remark}[]{Remark}

   \makeatletter
\def\proof{\@ifnextchar[{\@oproof}{\@nproof}}
\def\@oproof[#1][#2]{\trivlist\item[\hskip\labelsep\textit{#2 Proof of\
#1.}~]\ignorespaces}
\def\@nproof{\trivlist\item[\hskip\labelsep\textit{Proof.}~]\ignorespaces}

\makeatother

\begin{document}
\title[On Hurwitz zeta function and Lommel functions]{On Hurwitz zeta function and Lommel functions}
\author{Atul Dixit and Rahul Kumar}
\address{Department of Mathematics, Indian Institute of Technology, Gandhinagar, Palaj, Gandhinagar 382355, Gujarat, India}\email{adixit@iitgn.ac.in; rahul.kumr@iitgn.ac.in}
\dedicatory{Dedicated to Professor Bruce C. Berndt on the occasion of his 80th birthday}
\begin{abstract}
We obtain a new proof of Hurwitz's formula for the Hurwitz zeta function $\zeta(s, a)$ beginning with Hermite's formula. The aim is to reveal a nice connection between $\zeta(s, a)$ and a special case of the Lommel function $S_{\mu, \nu}(z)$. This connection is used to rephrase a modular-type transformation involving infinite series of Hurwitz zeta function in terms of those involving Lommel functions.
\end{abstract}
\maketitle
\section{Introduction}\label{intro}
The Hurwitz zeta function $\zeta(s, a)$ is defined for Re$(s)>1$ and $a\in\mathbb{C}\backslash\{x\in\mathbb{R}: x\leq 0\}$ by \cite[p.~36]{titch}
\begin{equation*}
\zeta(s, a):=\sum_{n=0}^{\infty}\frac{1}{(n+a)^{s}}.
\end{equation*}
It is well-known that for $\zeta(s, a)$ can be analytically continued to the entire $s$-complex plane except for a simple pole at $s=1$ with residue $1$, and that $\zeta(s, 1)=\zeta(s)$. 

One of the fundamental results in the theory of $\zeta(s, a)$ is the following formula of Hurwitz \cite[p.~37, Equation (2.17.3)]{titch}.
\begin{theorem}\label{fehurthm}
For $0<a\leq 1$ and $\textup{Re}(s)<0$,
\begin{equation}\label{fehur}
\zeta(s, a)=\frac{2\Gamma(1-s)}{(2\pi)^{1-s}}\left\{\sin\left(\frac{1}{2}\pi s\right)\sum_{n=1}^\infty \frac{\cos(2\pi na)}{n^{1-s}}+\cos\left(\frac{1}{2}\pi s\right)\sum_{n=1}^\infty \frac{\sin(2\pi na)}{n^{1-s}}\right\}.
\end{equation}
The above result also holds\footnote{See \cite[p.~257, Theorem 12.6]{apostol-1998a}.} for \textup{Re}$(s)<1$ if $0<a<1$.
\end{theorem}
We note that when $a=1$, the above formula reduces to the functional equation of $\zeta(s)$ \cite[p.~13, Equation (2.1.1)]{titch} for $\textup{Re}(s)<0$. which can then be seen to be true for all $s\in\mathbb{C}$ by analytic continuation.

Several proofs of \eqref{fehur} are available in the literature. For example, Hurwitz himself obtained it by transforming the Mellin transform representation of $\zeta(s, a)$ as a loop integral and then evaluating the latter. This proof can be found, for example, in \cite[p.~37]{titch}. Berndt \cite[Section 5]{berndtrocky1972} found a short proof of \eqref{fehur} by using the boundedly convergent Fourier series of $\lfloor x\rfloor-x+\tfrac{1}{2}$. We refer the reader interested in knowing the various proofs of this formula to \cite{ktty} and the references therein (see also \cite{kanetsuk}). In \cite[Section 4]{ktty}, Kanemitsu, Tanigawa, Tsukada and Yoshimoto obtained a new proof of \eqref{fehur}. Their proof commences with employing \cite[Equation (4.1)]{ktty} (see also \cite[Equation (47)]{kanechaktsuk})
\begin{align*}
\zeta(s, a)&=\frac{1}{2}a^{-s}+\frac{a^{1-s}}{s-1}\nonumber\\
&\quad+\sum_{n=1}^{\infty}\left\{\frac{e^{-2\pi ina}}{(-2\pi ina)^{1-s}}\G(1-s, -2\pi ina)+\frac{e^{2\pi ina}}{(2\pi ina)^{1-s}}\G(1-s, 2\pi ina)\right\},
\end{align*}
which is a special case of the Ueno-Nishizawa formula \cite{uenonishizawa} and then invoking the Fourier series of the Dirac-delta function $\delta(s)$. 

The aim of this note is to give a yet another new proof of \eqref{fehur} beginning with Hermite's well-known formula for $\zeta(s, a)$ \cite[p. 609, Formula 25.11.29]{nist}, valid for $\mathrm{Re}(a)>0$ and $s\neq 1$:
\begin{align}\label{herm}
\zeta(s,a)=\frac{1}{2}a^{-s}+\frac{a^{1-s}}{s-1}+2\int_0^\infty\frac{\sin(s\tan^{-1}(x/a))\, dx}{(a^2+x^2)^{s/2}\left(e^{2\pi x}-1\right)}.
\end{align}
The novelty of this proof is that it reveals the connection between Hurwitz zeta function and the Lommel functions $s_{\mu, \nu}(z)$ and $S_{\mu, \nu}(z)$ which, to the best of our knowledge, seems to have been unnoticed before. The Lommel functions are defined by \cite[p.~346, equation (10)]{Watson}
\begin{align}\label{small lommel}
s_{\mu,\nu}(z)=\frac{z^{\mu+1}}{(\mu-\nu+1)(\mu+\nu+1)}{}_1F_2\left(1;\frac{1}{2}\mu-\frac{1}{2}\nu+\frac{3}{2},\frac{1}{2}\mu+\frac{1}{2}\nu+\frac{3}{2};-\frac{1}{4}z^2\right).
\end{align}
and \cite[p.~347, equation (2)]{Watson}
\begin{align}\label{lommeldef1}
S_{\mu,\nu}(z)&=s_{\mu,\nu}(z)+\frac{2^{\mu-1}\G\left(\frac{\mu-\nu+1}{2}\right)
\G\left(\frac{\mu+\nu+1}{2}\right)}{\sin(\nu\pi)}\\
&\quad\quad\quad\quad\quad\times\left\{\cos\left(\frac{1}{2}(\mu-\nu)\pi\right)
J_{-\nu}(z)-\cos\left(\frac{1}{2}(\mu+\nu)\pi\right)J_{\nu}(z)\right\}\notag
\end{align}
for $\nu\notin\mathbb{Z}$, and
\begin{align}\label{lommeldefint}
S_{\mu,\nu}(z)&=s_{\mu,\nu}(z)+2^{\mu-1}\G\left(\frac{\mu-\nu+1}{2}\right)\G\left(\frac{\mu+\nu+1}{2}\right)\\
&\quad\quad\quad\quad\quad\times\left\{\sin\left(\frac{1}{2}(\mu-\nu)\pi\right)
J_{\nu}(z)-\cos\left(\frac{1}{2}(\mu-\nu)\pi\right)Y_{\nu}(z)\right\}\notag
\end{align}
for $\nu\in\mathbb{Z}$, where $J_{\nu}(z)$ and $Y_{\nu}(z)$ are Bessel functions of the first and second kinds respectively. The Lommel functions are the solutions of an inhomogeneous form of the Bessel differential equation \cite[p.~345]{Watson}, namely,
\begin{equation*}
z^2\frac{d^{2}y}{dz^{2}}+z\frac{dy}{dz}+(z^2-\nu^2)y=z^{\mu+1}.
\end{equation*}
Lommel functions arise in mathematics, for example, in the theory of positive trigonometric sums\cite{koulam}. Outside of mathematics, Lommel functions have been found to be very useful in physics as well as mathematical physics. See, for example, \cite{assjsv, goldstein, sitzer, thomaschan}. Lewis \cite{lewisinvent} studied a special case of $S_{\mu, \nu}(z)$, that is,
\begin{equation}\label{mathcalc}
\mathcal{C}_{s}(z)=\sqrt{z}\Gamma(2s+1)S_{-2s-\frac{1}{2},\frac{1}{2}}(z),
\end{equation}
and represented it in terms of the incomplete gamma function. Lewis and Zagier \cite{lewzag} represented the period functions for Maass wave forms with spectral parameter $s$ in terms of an infinite series of $\mathcal{C}_s(z)$, and in the course of which they gave different representations for this special case of the Lommel function. See \cite[p.~214, Proposition 1]{lewzag}.

In the present work, we require a new integral representation for this special case of the Lommel function $S_{\mu, \nu}(z)$ which, to the best of our knowledge, does not seem to have been explicitly stated anywhere including \cite{lewzag}. This is derived in Lemma \ref{lommel}. 

Another ingredient needed in our proof of \eqref{fehur} is a recent result of Maširević \cite[Theorem 2.1]{mas} (see also \cite[p.~176, Theorem 5.23]{barmaspog}) which states that for all $m\in\mathbb{N}\cup\{0\},\ \nu\in\mathbb{R},\ x\in(0,2\pi)$ and $\mu>\max\left\{-\nu-1,\nu-2,-\frac{1}{2}\right\}$,
\begin{align}\label{mas result}
\sum_{k=1}^\infty\frac{s_{\mu,\nu}(kx)}{k^{2m+\mu+1}}&=\frac{x^{\mu+1}}{4}\Gamma\left(\frac{1+\mu-\nu}{2}\right)\Gamma\left(\frac{1+\mu+\nu}{2}\right)\nonumber\\
&\quad\times\Bigg(\frac{(-1)^m\pi}{2\Gamma(m+1+(\mu-\nu)/2)\Gamma(m+1+(\mu+\nu)/2)}\left(\frac{x}{2}\right)^{2m-1}\nonumber\\
&\qquad+\sum_{n=0}^m\frac{(-1)^n\zeta(2m-2n)}{\Gamma(n+1+(1+\mu-\nu)/2)\Gamma(n+1+(1+\mu+\nu)/2)}\left(\frac{x}{2}\right)^{2n}\Bigg).
\end{align}

\section{A new proof of Hurwitz's formula using Hermite's formula \eqref{herm}}\label{newproof}
Here we prove Theorem \ref{fehurthm}. To do that, however, we first need a lemma which evaluates an integral in terms of the Lommel function $S_{\mu, \nu}(z)$. This lemma seems to be new.
\begin{lemma}\label{lommel}
Let the Lommel function $S_{\mu, \nu}(z)$ be defined in \eqref{lommeldef1} and \eqref{lommeldefint}. For $k\in\mathbb{N}$ and $a>0$, we have 
\begin{align}\label{lommeleqn}
\int_0^\infty \frac{e^{-2\pi kx}\sin(s\tan^{-1}(x/a))}{(a^2+x^2)^{s/2}}\ dx=s\sqrt{a}(2\pi k)^{s-\frac{1}{2}}S_{-s-\frac{1}{2}, \frac{1}{2}}(2\pi a k).
\end{align}
\end{lemma}
\begin{proof}
We first prove the above result for Re$(s)<0$ and then extend it to all $s\in\mathbb{C}$ by analytic continuation.

Using the inverse Mellin transform representation of the exponential function, for $c_1:=\mathrm{Re}(\xi)>0$ and $k>0$,
\begin{align}\label{mt 1}
\frac{1}{2\pi i}\int_{(c_1)}\frac{\Gamma(\xi)}{(2\pi k)^\xi}y^{-\xi}\ d\xi=e^{-2\pi ky},
\end{align}
where here, and throughout the paper, $\int_{(d)}$ will always mean the integral $\int_{d-i\infty}^{d+i\infty}$.

Also, from \cite[p.~193, Formula 5.19]{ober}, for $-1<c_2:=\mathrm{Re}(\xi)<\mathrm{Re}(s)$,
\begin{align}\label{mt 2}
\frac{1}{2\pi i}\int_{(c_2)}\frac{\Gamma(s-\xi)\Gamma(\xi)}{\Gamma(s)}\sin\left(\frac{\pi \xi}{2}\right)a^{\xi-s}y^{-\xi}\ d\xi=\frac{\sin\left(s\tan^{-1}\left(\frac{y}{a}\right)\right)}{(y^2+a^2)^{\frac{s}{2}}}.
\end{align}
From \eqref{mt 1}, \eqref{mt 2} and Parseval's identity \cite[p.~82, Equation (3.1.11)]{kp}
\begin{equation*}
\int_{0}^{\infty}g(x)h(x)\, dx=\frac{1}{2\pi i}\int_{(c)}\mathfrak{G}(1-s)\mathfrak{H}(s)\, ds,
\end{equation*}
where $\mathfrak{G}$ and $\mathfrak{H}$ are Mellin transforms of $g$ and $h$ respectively, for $c:=\mathrm{Re}(\xi)<\mathrm{Re}(s)$ and $-1<\mathrm{Re}(\xi)<1$,
{\allowdisplaybreaks\begin{align*}
2\int_0^\infty \frac{e^{-2\pi kx}\sin(s\tan^{-1}(x/a))}{(a^2+x^2)^{s/2}}\ dx&=\frac{2}{2\pi i}\int_{(c)}\frac{\Gamma(\xi)\Gamma(1-\xi)\Gamma(s-\xi)}{(2\pi k)^{1-\xi}\Gamma(s)}\sin\left(\frac{\pi\xi}{2}\right)a^{\xi-s}\ d\xi\nonumber\\
&=\frac{a^{-s}}{\Gamma(s)}\frac{1}{2\pi i}\int_{(c)}\frac{\pi\Gamma(s-\xi)}{\cos\left(\frac{\pi\xi}{2}\right)}(2\pi k)^{\xi-1}a^{\xi}\ d\xi\nonumber\\
&=\frac{a^{-s}}{\Gamma(s)}\frac{1}{2\pi i}\int_{(c)}\G\left(\frac{1+\xi}{2}\right)\G\left(\frac{1-\xi}{2}\right)\Gamma(s-\xi)(2\pi k)^{\xi-1}a^{\xi}\ d\xi,
\end{align*}}
where we used the reflection formula for the gamma function. Replacing $\xi$ by $-\xi$, we see that for $c:=\mathrm{Re}(\xi)>-\mathrm{Re}(s)$ and $-1<\mathrm{Re}(\xi)<1$,
\begin{align}\label{inter}
&2\int_0^\infty \frac{e^{-2\pi kx}\sin(s\tan^{-1}(x/a))}{(a^2+x^2)^{s/2}}\ dx=\frac{a^{-s}}{2\pi k\Gamma(s)}\frac{1}{2\pi i}\int_{(c)}\G\left(\frac{1+\xi}{2}\right)\G\left(\frac{1-\xi}{2}\right)\Gamma(s+\xi)(2\pi ak)^{-\xi}\ d\xi.
\end{align}
To evaluate the integral on the right-hand side of the above equation, first consider
\begin{align}\label{a110}
I_s(z):=\frac{1}{2\pi i}\int_{(c)}\G\left(\frac{1+\xi}{2}\right)\G\left(\frac{1-\xi}{2}\right)\Gamma(s+\xi)z^{-\xi}\ d\xi,
\end{align}
where $c=\textup{Re}(\xi)$. 

We evaluate the above integral first for $|z|<1, z\notin(-1,0]$ and then extend it later to all $z\in\mathbb{C}\backslash(-\infty, 0]$ by analytic continuation. Consider the contour formed by the line segments $[c-iT,c+iT],\ [c+iT,-\lambda+iT],\ [-\lambda+iT,-\lambda-iT]$ and $[-\lambda-iT,c-iT]$, where $\lambda\notin\mathbb{Z}, \l>1$. Observe that the integrand on the right-hand side of \eqref{a110} has simple poles at $\xi=-s-m, 0\leq m\leq\lfloor\l-\textup{Re}(s)\rfloor,$ and $\xi=-2n-1$, where $0\leq n\leq\lfloor\frac{\lambda-1}{2}\rfloor$ due to $\Gamma(s-\xi)$ and $\Gamma\left(\frac{1+\xi}{2}\right)$ respectively. (Note that since Re$(s)<0$, there will not be any pole of order $2$.) The residues of the integrand at these poles can be easily calculated to be $\frac{(-1)^mz^{m+s}\pi}{m!\cos\left(\frac{\pi}{2}(m+s)\right)}$ and $2(-1)^nz^{2n+1}\Gamma(-1-2n+s)$ respectively. By employing Stirling's formula in a vertical strip $\alpha\leq c\leq\beta$ \cite[p.~141, Formula \textbf{5.11.9}]{nist}, namely,
\begin{equation}\label{strivert}
|\Gamma(c+iT)|=(2\pi)^{\tf{1}{2}}|T|^{c-\tf{1}{2}}e^{-\tf{1}{2}\pi |T|}\left(1+O\left(\frac{1}{|T|}\right)\right),
\end{equation}
as $|T|\to\infty$, we see that the integrals along horizontal segments go to $0$ as $T\to\infty$. Hence by Cauchy's residue theorem, we obtain 
\begin{align}\label{cauchy}
I_{s}(z)&=\pi z^s\sum_{m=0}^{\lfloor\lambda-\textup{Re}(s)\rfloor} \frac{(-z)^m}{m!\cos\left(\frac{\pi}{2}(m+s)\right)}+2z\sum_{n=0}^{\lfloor\frac{\lambda-1}{2}\rfloor}(-1)^nz^{2n}\Gamma(-1-2n+s)\nonumber\\
&\quad+\frac{1}{2\pi i}\int_{(-\lambda)}\Gamma\left(\frac{1+\xi}{2}\right)\Gamma\left(\frac{1-\xi}{2}\right)\Gamma(s+\xi)z^{-\xi}\ d\xi.
\end{align}  
Next, we show that as $\lambda\rightarrow\infty$,
\begin{align}\label{integral zero}
\frac{1}{2\pi i}\int_{(-\lambda)}\Gamma\left(\frac{1+\xi}{2}\right)\Gamma\left(\frac{1-\xi}{2}\right)\Gamma(s+\xi)z^{-\xi}\ d\xi\rightarrow 0.
\end{align}
By \eqref{strivert}, we find that as $|t|\rightarrow\infty$,
\begin{align}\label{gamma1}
\left|\Gamma\left(\frac{1\pm\lambda-it}{2}\right)\right|=O_{\l}\left(|t|^{\pm\lambda/2}e^{-\frac{\pi}{4}|t|}\right), 
\end{align}
and
\begin{align}\label{gamma3}
|\Gamma(s-\lambda+it)|=O\left(|(t+\mathrm{Im}(s))|^{-\lambda+\mathrm{Re}(s)-1/2}e^{-\frac{\pi}{2}|(t+\mathrm{Im}(s))|}\right).
\end{align}
Upon making change of variable $\xi=-\lambda+it$ and then using \eqref{gamma1} and \eqref{gamma3}, we see that
\begin{align*}
&\left|\int_{(-\lambda)}\Gamma\left(\frac{1+\xi}{2}\right)\Gamma\left(\frac{1-\xi}{2}\right)\Gamma(s+\xi)z^{-\xi}\ d\xi\right|\nonumber\\
&=\left|\frac{1}{i}\int_{-\infty}^\infty\Gamma\left(\frac{1-\lambda+it}{2}\right)\Gamma\left(\frac{1+\lambda-it}{2}\right)\Gamma(s-\lambda+it)z^{\lambda-it}\ dt\right|\nonumber\\
&=|z|^\lambda \int_{-M}^MO(1)\ dt+|z|^\lambda\int_{|t|\geq M}O\left(\left|\left(t+\mathrm{Im}(s)\right)\right|^{-\lambda+\mathrm{Re}(s)-1/2}e^{-\frac{\pi}{2}|t|-\frac{\pi}{2}|(t+\mathrm{Im}(s))|}\right)\ dt\nonumber\\
&=O(|z|^\lambda),
\end{align*}
where $M$ is a large enough positive real number. Since $|z|<1$, as $\lambda\rightarrow\infty$, we arrive at \eqref{integral zero}. Therefore by \eqref{cauchy} and \eqref{integral zero}, for $|z|<1, z\notin(-1,0]$,
\begin{align}\label{a12}
I_{s}(z)
=\pi z^s\sum_{m=0}^{\infty} \frac{(-z)^m}{m!\cos\left(\frac{\pi}{2}(m+s)\right)}+2z\sum_{n=0}^{\infty}(-1)^nz^{2n}\Gamma(-1-2n+s).
\end{align}
Now observe that both sides of above equation are analytic, as functions of $z$, in $\mathbb{C}\backslash(-\infty,0]$. Therefore \eqref{a12} holds for all $z\in\mathbb{C}\backslash(-\infty, 0]$ by analytic continuation. Hence letting $z=2\pi ak, a>0, k\in\mathbb{N}$, in \eqref{a12} and simplifying the resulting first sum by splitting it into two sums, one over even $m$ and other over odd $m$ and by rephrasing the resulting second sum into a ${}_1F_{2}$ using the reflection and duplication formulas of the gamma function, we obtain
\begin{align}\label{a13}
I_{s}(2\pi a k)
=\frac{2\pi(2\pi ak)^{s}}{\sin(\pi s)}\sin\left(\frac{\pi}{2}(4ak+s)\right)+4\pi ak\Gamma(s-1){}_1F_2\left(1;1-\frac{s}{2},\frac{3-s}{2};-a^2\pi^2k^2\right).
\end{align}
Equation \eqref{lommeleqn} now follows from \eqref{small lommel}, \eqref{lommeldef1}, \eqref{inter}, \eqref{a110} and \eqref{a13}. This completes the proof for Re$(s)<0$. Now using an argument similar to that in \cite[p.~269-270]{whitwat}, one can show that the left-hand side of \eqref{lommeleqn} is an entire function of $s$. The right-hand side of \eqref{lommeleqn} is also analytic in the whole $s$-complex plane except for possible poles at $s=1, 2, 3\cdots$. However, as shown in \cite[p.~347-349]{Watson}, the Lommel function $S_{\mu, \nu}(z)$ has a limit when $\mu+\nu$ or $\mu-\nu$ are odd negative integers. Since the positive integer values of $s$ render the $\mu+\nu$ and $\mu-\nu$ of our special case of the Lommel function, namely, $S_{-s-\frac{1}{2}, \frac{1}{2}}(2\pi a k)$, to fall precisely in this category, we see that $s=1, 2, 3, \cdots$ are indeed removable singularities of the right-hand side. Therefore both sides of \eqref{lommeleqn} are entire functions of $s$, and hence the equality follows for all $s\in\mathbb{C}$ by analytic continuation.
\end{proof}
We are now ready to prove Theorem \ref{fehurthm}.

\begin{proof}[Theorem \textup{\ref{fehurthm}}][]
The proof is divided into two cases.\\

\noindent
\textbf{Case 1:} $0<a<1$.

\noindent
We first prove the result for $s<0$ and then extend it by analytic continuation. For $a>0$ and $s\neq 1$, from \eqref{herm},
\begin{align}\label{above}
\zeta(s,a)
=\frac{1}{2}a^{-s}+\frac{a^{1-s}}{s-1}+2\sum_{k=1}^\infty\int_0^\infty \frac{e^{-2\pi kx}\sin(s\tan^{-1}(x/a))}{(a^2+x^2)^{s/2}}\ dx,
\end{align}
where the interchange of the order of summation and integration is justified by absolute and uniform convergence (see, for example, \cite[p.~30, Theorem 2.1]{temme}). Invoking Lemma \ref{lommel} in \eqref{above}, we have, for $a>0$ and $s\neq 1$,
\begin{align}\label{befmas}
\zeta(s, a)&=\frac{1}{2}a^{-s}+\frac{a^{1-s}}{s-1}+2s\sqrt{a}(2\pi)^{s-\frac{1}{2}}\sum_{k=1}^\infty\frac{S_{-s-\frac{1}{2}, \frac{1}{2}}(2\pi a k)}{k^{\frac{1}{2}-s}}\nonumber\\
&=\frac{1}{2}a^{-s}+\frac{a^{1-s}}{s-1}+\frac{1}{\G(s)}\sum_{k=1}^\infty\bigg\{2\sqrt{a}(2\pi)^{s-\frac{1}{2}}\G(s+1)\frac{s_{-s-\frac{1}{2}, \frac{1}{2}}(2\pi a k)}{k^{\frac{1}{2}-s}}\nonumber\\
&\qquad\qquad\qquad\qquad\qquad\qquad\quad+\frac{(2\pi)^{s}}{\sin(\pi s)}\frac{\sin\left(\frac{\pi}{2}(4ak+s)\right)}{k^{1-s}} \bigg\},
\end{align}
where in the second step, we used \eqref{lommeldef1}. This is the first instance where the infinite series on the right-hand side of \eqref{fehur} makes its conspicuous presence.

Now let $m=0,\ \mu=-s-\frac{1}{2},\ \nu=\frac{1}{2}$ and $x=2\pi a$ in \eqref{mas result} and then use $\zeta(0)=-1/2$ to deduce that for $s<0$ and $0<a<1$,
\begin{align*}
\sum_{k=1}^\infty\frac{s_{-s-\frac{1}{2}, \frac{1}{2}}(2\pi a k)}{k^{\frac{1}{2}-s}}&=\frac{(2\pi a)^{\frac{1}{2}-s}}{4}\Gamma\left(-\frac{s}{2}\right)\Gamma\left(\frac{1-s}{2}\right)\left(\frac{\pi}{2\pi a\Gamma\left(\frac{1-s}{2}\right)\Gamma\left(1-\frac{s}{2}\right)}+\frac{-1/2}{\Gamma\left(1-\frac{s}{2}\right)\Gamma\left(\frac{3-s}{2}\right)}\right)\nonumber\\
&=\frac{(2\pi a)^{\frac{1}{2}-s}}{4}\frac{\sqrt{\pi}\Gamma(-s)}{2^{-s-1}}\left(\frac{2^{-s}}{2a\sqrt{\pi}\Gamma(1-s)}-\frac{2^{1-s}}{2\sqrt{\pi}\Gamma(2-s)}\right),
\end{align*}
where in the last step, we used the duplication formula for the gamma function twice. Thus,
\begin{align}\label{masi}
\sum_{k=1}^\infty\frac{s_{-s-\frac{1}{2}, \frac{1}{2}}(2\pi a k)}{k^{\frac{1}{2}-s}}=\frac{-1}{2s\sqrt{a}(2\pi)^{s-\frac{1}{2}}}\left(\frac{1}{2}a^{-s}+\frac{a^{1-s}}{s-1}\right).
\end{align}
Hence from \eqref{befmas} and \eqref{masi},
\begin{align*}
\zeta(s, a)= \frac{(2\pi)^{s}}{\G(s)\sin(\pi s)}\sum_{k=1}^\infty\frac{\sin\left(\frac{\pi}{2}(4ak+s)\right)}{k^{1-s}},
\end{align*}
which readily gives \eqref{fehur}. This completes the proof of Theorem \ref{fehurthm} for $s<0$. Since both sides of \eqref{fehur} are analytic for Re$(s)<0$, we conclude that \eqref{fehur} is valid for Re$(s)<0$.

If $0<a<1$, note that in addition to being absolutely convergent for Re$(s)<0$, the series on the right-hand side of \eqref{fehur} are conditionally convergent for $0<\textup{Re}(s)<1$ whence we see that for $0<a<1$, the result \eqref{fehur} actually holds for Re$(s)<1$.\\

\noindent
\textbf{Case 2:} $a=1$.

\noindent
We first prove the result for $s<-1$ and then extend it to all complex $s$ by analytic continuation.
From \eqref{befmas} and the fact that $\zeta(s, 1)=\zeta(s)$ for all $s\in\mathbb{C}$, we have
\begin{align}\label{befmas1}
\zeta(s)&=\frac{1}{2}+\frac{1}{s-1}+\frac{1}{\G(s)}\sum_{k=1}^\infty\bigg\{2(2\pi)^{s-\frac{1}{2}}\G(s+1)\frac{s_{-s-\frac{1}{2}, \frac{1}{2}}(2\pi k)}{k^{\frac{1}{2}-s}}
+\frac{(2\pi)^{s}}{\sin(\pi s)}\frac{\sin\left(\frac{\pi}{2}(4k+s)\right)}{k^{1-s}} \bigg\}.
\end{align}
We now wish to evaluate in closed form the series $\displaystyle\sum_{k=1}^\infty\frac{s_{-s-\frac{1}{2}, \frac{1}{2}}(2\pi k)}{k^{\frac{1}{2}-s}}$, which is indeed convergent for $s<-1$ since $\sum_{k=1}^{\infty}k^{s-1}\sin\left(\frac{\pi}{2}(4k+s)\right)$ converges for $s<-1$.

However, one should be careful as \eqref{mas result} cannot be applied here. This is because, it requires $x\in(0, 2\pi)$, whereas here we need $x$ in \eqref{mas result} to be $2\pi$. Thankfully, Maširević has also obtained the following result \cite[Theorem 2.2]{mas} where $0\leq x\leq 2\pi, m\in\mathbb{N}$ and $\mu>0$:
\begin{align}\label{mas result 2}
\sum_{k=1}^\infty\frac{s_{\mu-\frac{3}{2},\frac{1}{2}}(kx)}{k^{2m+\mu-\frac{1}{2}}}=\frac{1}{2}(-1)^{m-1}x^{2m+\mu-\frac{1}{2}}\G(\mu-1)\left(\frac{-\pi}{x\G(2m+\mu)}+2\sum_{n=0}^{m}\frac{(-1)^{n-1}\zeta(2n)}{\G(2m+\mu+1-2n)x^{2n}}\right).
\end{align}
Even though $x$ can equal $2\pi$ in the above result, note that $m$ has to be a natural number, whereas, to evaluate the series $\displaystyle\sum_{k=1}^\infty\frac{s_{-s-\frac{1}{2}, \frac{1}{2}}(2\pi k)}{k^{\frac{1}{2}-s}}$ using \eqref{mas result 2}, we would require $m=0$. To circumvent this problem, we first employ the well-known result \cite[p.~946, Formula \textbf{8.575.1}]{grn},
\begin{equation}\label{lommelcontiguous}
s_{\mu+2, \nu}(z)=z^{\mu+1}-[(\mu+1)^2-\nu^2]s_{\mu, \nu}(z).
\end{equation}
Use \eqref{lommelcontiguous} with $\mu=-s-5/2$, $\nu=1/2$ and $z=2\pi k$ so that
\begin{align}\label{modilommel0}
\sum_{k=1}^\infty\frac{s_{-s-\frac{1}{2}, \frac{1}{2}}(2\pi k)}{k^{\frac{1}{2}-s}}=\sum_{k=1}^{\infty}\frac{(2\pi k)^{-s-\frac{3}{2}}}{k^{\frac{1}{2}-s}}-\left\{\left(s+\frac{3}{2}\right)^2-\frac{1}{4}\right\}\sum_{k=1}^\infty\frac{s_{-s-\frac{5}{2}, \frac{1}{2}}(2\pi k)}{k^{\frac{1}{2}-s}}.
\end{align}
We now transform the second series on the right-hand side of the above equation using \eqref{mas result 2}. Before we do that, however, we need the well-known result $\zeta(0)=-1/2$, which can be proved \emph{without} using the functional equation of $\zeta(s)$ so that circular reasoning is avoided. For example, one can put $s=0$ in the following formula \cite[p.~14, Equation (2.1.4)]{titch} 
\begin{equation*}
\zeta(s)=s\int_{1}^{\infty}\frac{[x]-x+1/2}{x^{s+1}}\, dx+\frac{1}{s-1}+\frac{1}{2}\hspace{8mm}(\textup{Re}(s)>-1),
\end{equation*}
to conclude that $\zeta(0)=-1/2$.

Let $m=1$, $x=2\pi$ and $\mu=-s-1$ in \eqref{mas result 2} so that for $s<-1$,
\begin{align}\label{modilommel}
\sum_{k=1}^\infty\frac{s_{-s-\frac{5}{2}, \frac{1}{2}}(2\pi k)}{k^{\frac{1}{2}-s}}=\frac{1}{2}(2\pi)^{\frac{1}{2}-s}\G(-s-2)\left(-\frac{1}{2\G(1-s)}+\frac{1}{\G(2-s)}+\frac{1}{12\G(-s)}\right),
\end{align}
where we used $\zeta(0)=-1/2$ and $\zeta(2)=\sum_{k=1}^{\infty}\frac{1}{k^2}=\pi^2/6$.

Substitute \eqref{modilommel} in \eqref{modilommel0} and simplify using the functional equation of the gamma function to arrive at
\begin{align}\label{a=1}
\sum_{k=1}^\infty\frac{s_{-s-\frac{1}{2}, \frac{1}{2}}(2\pi k)}{k^{\frac{1}{2}-s}}=-\frac{1}{2s(2\pi)^{s-\frac{1}{2}}}\left(\frac{1}{2}+\frac{1}{s-1}\right).
\end{align}
Comparing the above equation with \eqref{masi}, we see that \eqref{masi} holds for $a=1$ too.

Using \eqref{a=1} in \eqref{befmas1}, we arrive at
\begin{equation*}
\zeta(s)=\frac{2\G(1-s)}{(2\pi)^{1-s}}\sin\left(\frac{\pi s}{2}\right)\zeta(1-s)
\end{equation*}
for $s<-1$. The result then follows for all complex $s$ by analytic continuation.
\end{proof}

\section{A modular-type transformation involving the Lommel function $S_{-s-\frac{1}{2}, \frac{1}{2}}(z)$}
Modular-type transformations are the ones governed by the map $\a\to\b$, where $\a\b=1$. An equivalent way to say this using the language of modular forms is that they consist of functions which transform nicely under $z\to-1/z$, $\textup{Im}(z)>0$. But they may not transform nicely under $z\to z+1$, hence the nomenclature \emph{modular-type} transformations. For a detailed survey on modular-type transformations, the reader is referred to \cite{dixmathstu}.

The following modular-type transformation involving infinite series of Hurwitz zeta function was obtained by the first author in \cite[Theorem 1.4]{dixit} as a generalization of a transformation of Ramanujan \cite[Theorem 1.1]{dixit} on page $220$ of the Lost Notebook \cite{lnb}.
\begin{theorem}\label{mainnn}
Let $0<$ \textup{Re}$(s)<2$. Define $\varphi(s,x)$ by
\begin{equation*}
\varphi(s,x):=\zeta(s,x)-\frac{1}{2}x^{-s}+\frac{x^{1-s}}{1-s}.
\end{equation*}
If $\alpha$ and $\beta$ are any positive numbers such that $\alpha\beta=1$,
\begin{align*}
&\a^{\frac{s}{2}}\left(\sum_{n=1}^{\infty}\varphi(s,n\a)-\frac{\zeta(s)}{2\alpha^{s}}-\frac{\zeta(s-1)}{(s-1)\alpha }\right)=\b^{\frac{s}{2}}\left(\sum_{n=1}^{\infty}\varphi(s,n\b)-\frac{\zeta(s)}{2\beta^{s}}-\frac{\zeta(s-1)}{(s-1)\beta}\right)\nonumber\\
&=\frac{8(4\pi)^{\frac{s-4}{2}}}{\Gamma(s)}\int_{0}^{\infty}\Gamma\left(\frac{s-2+it}{4}\right)\Gamma\left(\frac{s-2-it}{4}\right)
\Xi\left(\frac{t+i(s-1)}{2}\right)\nonumber\\
&\qquad\qquad\qquad\times\Xi\left(\frac{t-i(s-1)}{2}\right)\frac{\cos\left( \tf{1}{2}t\log\a\right)}{s^2+t^2}\, dt,
\end{align*}
where $\Xi(t):=\xi\left(\frac{1}{2}+it\right)$ with $\xi(s)=\frac{1}{2}s(s-1)\pi^{-\frac{1}{2}s}\Gamma(\tfrac{1}{2}s)\zeta(s)$.
\end{theorem}
In view of the first equality in \eqref{befmas}, the modular-type transformation in the above result can be rephrased in the following form.
\begin{corollary}\label{modlommel}
Let $0<\textup{Re}(s)<2$ and let $\sigma_{s}(n)=\sum_{d|n}d^s$. Let $S_{\mu, \nu}(z)$ be defined in \eqref{lommeldef1}. If $\alpha$ and $\beta$ are any positive numbers such that $\alpha\beta=1$,
\begin{align*}
&\a^{\frac{s}{2}}\left(2s(2\pi)^{s-\frac{1}{2}}\sqrt{\a}\sum_{m=1}^{\infty}\sigma_{1-s}(m)m^{s-\frac{1}{2}}S_{-s-\frac{1}{2}, \frac{1}{2}}(2\pi m\a)-\frac{\zeta(s)}{2\alpha^{s}}-\frac{\zeta(s-1)}{(s-1)\alpha}\right)\nonumber\\
&=\b^{\frac{s}{2}}\left(2s(2\pi)^{s-\frac{1}{2}}\sqrt{\b}\sum_{m=1}^{\infty}\sigma_{1-s}(m)m^{s-\frac{1}{2}}S_{-s-\frac{1}{2}, \frac{1}{2}}(2\pi m\b)-\frac{\zeta(s)}{2\b^{s}}-\frac{\zeta(s-1)}{(s-1)\b}\right).
\end{align*}
\end{corollary}
\begin{proof}
The result follows at once if we observe that from \eqref{befmas},
\begin{align*}
\sum_{n=1}^{\infty}\varphi(s,n\a)&=2s(2\pi)^{s-\frac{1}{2}}\sqrt{\a}\sum_{n=1}^{\infty}\sqrt{n}\sum_{k=1}^\infty\frac{S_{-s-\frac{1}{2}, \frac{1}{2}}(2\pi nk\a)}{k^{\frac{1}{2}-s}}\nonumber\\
&=2s(2\pi)^{s-\frac{1}{2}}\sqrt{\a}\sum_{n, k=1}^{\infty}n^{1-s}(nk)^{s-\frac{1}{2}}S_{-s-\frac{1}{2}, \frac{1}{2}}(2\pi nk\a)\nonumber\\
&=2s(2\pi)^{s-\frac{1}{2}}\sqrt{\a}\sum_{m=1}^{\infty}\sigma_{1-s}(m)m^{s-\frac{1}{2}}S_{-s-\frac{1}{2}, \frac{1}{2}}(2\pi m\a).
\end{align*}
\end{proof}
\begin{remark}
The series in Corollary \ref{modlommel} should be compared with the series considered by Lewis and Zagier in \cite[Equation (2.11)]{lewzag}, namely, $\sum_{n=1}^{\infty}n^{s-1/2}A_n\mathcal{C}_{s}(2\pi na)$, where $\mathcal{C}_{s}(z)$ is defined by \eqref{mathcalc}.
\end{remark}
\begin{center}
\textbf{Acknowledgements}
\end{center}

\noindent
The first author's research is partially supported by the SERB MATRICS grant MTR/2018/000251. He sincerely thanks SERB for the support.


\begin{thebibliography}{00}
\bibitem{apostol-1998a}
Apostol, T.M.: Introduction to Analytic Number Theory, Springer-Verlag, New York (1998).

\bibitem{assjsv}
N.~T.~Adelman, Y.~Stavsky and E.~Segal, \emph{Axisymmetric vibrations of radially polarized piezoelectric ceramic cylinders}, Journal of Sound and Vibration~\textbf{38}, no.~2 (1975), 245--254.

\bibitem{barmaspog}
\'{A}.~Baricz, D.~J.~Maširević and T.~K.~Pogany, \emph{Series of Bessel and Kummer-type functions}, Lecture Notes in Mathematics, 2207, Springer, Cham, 2017.

\bibitem{berndtrocky1972}
B.~C.~Berndt, \emph{On the Hurwitz zeta-function}, Rocky Mountain J.~Math.~\textbf{2} No. 1 (1972), 151--157.

\bibitem{dixit}
A.~Dixit, \emph{Analogues of a transformation formula of Ramanujan}, Int. J. Number Theory~\textbf{7}, No. 5 (2011), 1151-1172.

\bibitem{dixmathstu}
A.~Dixit, \emph{Modular-type transformations and integrals involving the Riemann $\Xi$-function}, Math.~Student~\textbf{87} Nos.~3-4 (2018), 47--59.

\bibitem{goldstein}
S.~Goldstein, \emph{On the Vortex Theory of Screw Propellers}, Proc. R. Soc. Lond. A~\textbf{123} (1929), 440--465.

\bibitem{grn} I.S. Gradshteyn and I. M. Ryzhik, eds., \textit{Table of Integrals, Series, and products,} 7th ed., Edited by A.~Jeffrey and D.~Zwillinger, Academic Press, New York, 2007.

\bibitem{ktty}
S.~Kanemitsu, Y.~Tanigawa, H.~Tsukada and M.~Yoshimoto, \emph{Contributions to the theory of the Hurwitz zeta-function}, Hardy-Ramanujan J.~\textbf{30} (2007), 31--55.

\bibitem{kanetsuk}
S.~Kanemitsu and H.~Tsukada, \emph{Contributions to the Theory of Zeta-Functions. The Modular Relation Supremacy}, Series on Number Theory and its Applications, Vol. 10, World Scientific, New Jersey 2015.

\bibitem{kanechaktsuk}
S.~Kanemitsu, K.~Chakraborty and H.~Tsukada, \emph{Ewald expansions of a class of zeta-functions}, SpringerPlus~\textbf{5} No. 1 (2016).

\bibitem{koulam}
S.~Koumandos and M.~ Lamprecht, \emph{The zeros of certain Lommel functions}, Proc.~Amer.~Math.~Soc.~\textbf{140} No. 9 (2012), 3091--2100.

\bibitem{lewzag}
J.~Lewis and D.~Zagier, \emph{Period functions for Maass wave forms. \textup{I}}, Ann.~Math.~\textbf{153}, No. 1 (2001), 191--258.

\bibitem{lewisinvent}
J.~Lewis, \emph{Spaces of holomorphic functions equivalent to the even Maass cusp form}, Invent. Math.~\textbf{127} (1997), 271--306.

\bibitem{mas}
D. J. Maširević, \emph{Summations of Schlömilch series containing some Lommel functions of the first kind terms}, Integral Transforms Spec. Funct. \textbf{27}(2), 153-162.

\bibitem{nist} F. W. J. Olver, D. W. Lozier, R. F. Boisvert, and C. W. Clark, eds., \textit{NIST Handbook of Mathematical
Functions}, Cambridge University Press, Cambridge, 2010.

\bibitem{ober}
F.~Oberhettinger, \emph{Tables of Mellin Transforms}, Springer-Verlag, New York, 1974.

\bibitem{kp}
R.~B.~Paris and D.~Kaminski, \emph{Asymptotics and Mellin-Barnes Integrals},  Encyclopedia of Mathematics and its Applications, 85. Cambridge University Press, Cambridge, 2001.

\bibitem{lnb}
S.~Ramanujan, \emph{The Lost Notebook and Other Unpublished
Papers}, Narosa, New Delhi, 1988.

\bibitem{sitzer}
M.~R.~Sitzer, \emph{Stress distribution in rotating aeolotropic laminated
  heterogeneous disc under action of a time-dependent
loading}, Z. Angew. Math. Phys.~\text{36} (1985), 134--145.

\bibitem{temme}
N.~M.~Temme, \emph{Special functions: An introduction to the classical functions of mathematical physics}, Wiley-Interscience Publication, New York, 1996.

\bibitem{thomaschan}
B.~K.~Thomas and F.~T.~Chan, \emph{Glauber $e^{-} +$ He elastic scattering amplitude: A useful integral representation}, Phys. Rev. A~\textbf{8}, 252--262.

\bibitem{titch}
E.~C.~Titchmarsh, \emph{The Theory of the Riemann Zeta Function}, Clarendon Press, Oxford, 1986.

\bibitem{uenonishizawa}
K.~Ueno and M.~Nishizawa, \emph{Quantum groups and zeta-functions}, in Quantum Groups: Formalism and Applications, Lubkierski et al (eds), Proceedings of the Thirtieth Karpacz Winter School (Karpacz) (1994), Polish Sci. Publ. PWN, Warsaw, pp 115--126.

\bibitem{Watson} G.N. Watson, \textit{A Treatise on the Theory of Bessel Functions}, second ed., Cambridge University Press, London, 1994.

\bibitem{whitwat}
E.~T.~Whittaker and G.~N.~Watson, \emph{A Course of Modern Analysis}, Cambridge University Press, 1996.
\end{thebibliography}
\end{document}